\documentclass[11pt]{amsart}       % onecolumn (second format)
\usepackage[margin=1in]{geometry}
\usepackage{hyperref}
\usepackage{color}
\usepackage{graphicx}
\usepackage{latexsym}
\usepackage{enumerate}
\usepackage{cite}
\usepackage{paralist}
\usepackage{tikz}
\usetikzlibrary{decorations.pathmorphing}
\usetikzlibrary{decorations.markings}
\usetikzlibrary{shapes,positioning,matrix,arrows}
\usetikzlibrary{patterns}
\usepackage{enumerate}
\usepackage{amsmath,amsfonts,amsthm}
\usepackage{bm}
\newcommand{\Z}{\hbox{$\mathbb Z$}}
\newcommand{\Q}{\hbox{$\mathbb Q$}}
\newcommand{\R}{\hbox{$\mathbb R$}}

\newcommand{\PP}{\hbox{$\mathcal P$}}
\newcommand{\A}{\hbox{$\mathcal A$}}
\DeclareMathOperator{\sgn}{sgn}
\DeclareMathOperator{\sig}{sig}
\DeclareMathOperator{\cone}{cone}
\DeclareMathOperator{\conv}{conv}

\DeclareMathOperator{\coloneqq}{:=}
\newcommand{\PG}{\mathcal{P}_G}
\newcommand{\APG}{\mathcal{A}(\PG)}
\theoremstyle{definition}
\newtheorem{theorem}{Theorem}[section]
\newtheorem{corollary}[theorem]
{Corollary}
\newtheorem{lemma}[theorem]{Lemma}
\newtheorem{notation}[theorem]{Notation}
\newtheorem{proposition}[theorem]{Proposition}

\newtheorem{definition}[theorem]{Definition}

\newtheorem{observation}[theorem]{Observation}

\newtheorem{example}[theorem]{Example}

\newtheorem*{spthmA}{Theorems 3 and 5}{\bfseries\upshape}{\itshape}
\newtheorem*{spthmB}{Theorems 4 and 6}{\bfseries\upshape}{\itshape}

\begin{document}

\title{Quadratic Generated Normal Domains from Graphs
}


\author{Michael A. Burr}
\author{Drew J. Lipman  %etc.
}


\address{M. Burr \\ Dept of Mathematical Sciences, Clemson University \\
							burr2@clemson.edu}         %  \\
\address{D. Lipman \\ Dept of Mathematical Sciences, Clemson University \\
              djlipma@g.clemson.edu
}


%\maketitle

\begin{abstract}
Determining whether an arbitrary subring $R$ of $k[x_1^{\pm 1},\dots, x_n^{\pm 1}]$ is a normal domain is, in general, a nontrivial problem, even in the special case of a monomial generated domain.
In this paper, we provide a complete characterization of the normality and normalizations of quadratic-monomial generated domains.
For a quadratic-monomial generated domain $R$, we develop a combinatorial structure that assigns, to each quadratic monomial of the ring, an edge in a mixed signed, directed graph $G$, i.e., a graph with signed edges and directed edges.
We classify the normality and the normalizations of such rings in terms of a generalization of the combinatorial odd cycle condition on $G$.
\end{abstract}

\maketitle
%%%%%%%%%%%%%%%%%%%%%%%%%%      This is to remove the page number from the bottom of the first page %%%%%%%%%%%%%%%%%%%%%%%
%\thispagestyle{empty}

\section{Introduction}
\label{sec:Intro}
It is, in general, a nontrivial problem to determine whether an arbitrary subring of a Laurent polynomial ring $k[x_1^{\pm 1},\dots, x_n^{\pm 1}]$ over a field $k$ is a normal domain, even in the special case of monomial generated domains, see, e.g., \cite{BrunsKoch:2001,Vasconcelos:2005,HunekeSwanson:2006}.  Hibi and Ohsugi \cite{OhsugiHibi:1998} and Simis, Vasconcelos, and Villarreal \cite{Villarreal:1998} independently study the special case of quadratic-monomial generated subrings of the polynomial ring $k[x_1,\dots,x_n]$ (see also \cite{BrunsVillarreal:1997, Villarreal:1994, Villarreal:1995, Villarreal:1996}).  In their work, they construct a quadratic-monomial generated subring $k[G]$, called the {\em edge ring}\footnote{The edge ring is a different object than the edge ideal of a graph.  The edge ring and edge ideal are generated by the same elements, but one as a subring and the other as an ideal of $k[x_{1},\dots,x_n]$.  For more details on the edge ideal, see, for example, \cite{Villarreal:1994,Sullivant:2008} and the references included therein.  It is also distinct from the edge algebra (sometimes also called the edge ring), see, for example \cite{Estrada:2000,GitlerValencia:2005} and the references included therein.  There is another distinct concept called the edge ring, see, for example \cite{Zalavsky:2009}.}, from a graph $G$.  From the graph, they give a combinatorial characterization, called the odd cycle condition, of the normality of $k[G]$ in terms of $G$.  Moreover, when $k[G]$ is not a normal domain, they use the combinatorial data of $G$ to construct the normalization of $k[G]$.

In this paper, we generalize their work and provide a complete characterization of the normality of and normalizations for quadratic-monomial generated domains in the Laurent polynomial ring $k[x_1^{\pm 1},\dots,x_n^{\pm 1}]$.  This case is considerably more complicated than the situation in the previous work because the negative powers allow exponents to cancel.  In order to address these difficulties, we introduce new proofs; in particular, our proofs are are more graph-theoretic in nature than in the previous work.  Combinatorially, we study mixed signed, directed graphs, i.e., graphs with signed edges and directed edges and provide generalizations of the odd cycle condition and the combinatorial data of $G$ for the normalization of $k[G]$ in this case.

Since edge rings come equipped with an accessible combinatorial structure, these rings have been studied extensively in the commutative algebra setting, see, e.g., \cite{HibiMatsudaOhsugi:2016,HibiNishiyamaOhsugiShikama:2014,HibiKatthan:2014,HibiHigashitaniKimura:2014,TatakisThoma:2013,Matsui:2003,GitlerVillareal:2011,HerzogHibiZheng:2004,OhsugiHibi:2000,DAli:2015,HibiOhsugi:2006,Villarreal:2005,HibiMoriOhsugiShikama:2016,BermejoGarciaReyes:2015,BermejoGimenezSimis:2009,OhsugiHibi:2017,HibiKattan:2014}.  These papers classify commutative algebraic properties for edge rings and use edge rings to easily construct rings which are examples and non-examples for these properties.  We expect that our generalization of edge rings will naturally extend to these settings, and thereby extend the results and examples in these papers.  In particular, the study of Serre's conditions in \cite{HibiKattan:2014} has already been extended to this more general setting in \cite{LipmanBurr:2017} by providing a more general theorem with new examples.  Additionally, in the field of algebraic statistics, normal domains and semigroups have several applications, see, for example \cite{Sullivant:2010,SturmfelsWelker:2012} and the references included therein for examples.  By extending the collection of examples of normal and non-normal graphs, our results may have applications in these associated fields.

\subsection{Main Results}

Suppose that $R$ is a quadratic-monomial generated subring of $k[x_1^{\pm 1},\dots,x_n^{\pm 1}]$.  From $R$, we construct a mixed signed, directed graph $G$, i.e., the edges of $G$ are either directed or signed, such that the edge ring $k[G]$ associated to $G$ equals $R$.  Our first main result characterizes the normality of $R$ in terms of the combinatorial properties of $G$:

\begin{spthmA}
Let $G$ be a mixed signed, directed graph.  $k[G]$ is normal if and only if $G$ satisfies the generalized odd cycle condition, i.e., for any two disjoint cycles, each with an odd number of signed edges, between them there is either no path or a generalized alternating path, see Definition \ref{def:genaltpath}.
\end{spthmA}

When $k[G]$ is not normal, we call pairs of disjoint cycles exceptional if they fail the odd cycle condition, i.e., each cycle has an odd number of signed edges and between them there is a path, but no generalized alternating path.  For each pair $\Pi$ of exceptional cycles, we define a monomial $M_{\Pi}$ in Definition \ref{def:signedPi}.  These monomials then generate the normalization of $k[G]$.

\begin{spthmB}
Let $G$ be a mixed signed, directed graph.  Suppose that $\Pi_1,\dots,\Pi_q$ are the exceptional pairs of cycles in $G$.  Then, the normalization of $k[G]$ is generated as an algebra by $M_{\Pi_1},\dots,M_{\Pi_q}$ over $k[G]$
\end{spthmB}

\subsection{Outline of Paper}

The remainder of this paper is organized as follows:  In Section \ref{sec:BackAndNote}, we provide the background and notation used throughout the paper.  In Section \ref{sec:Algebraic}, we prove our main theorems for the edge ring of a signed graph.  In particular, we characterize when $k[G]$ is normal and provide the normalization when it is not normal.  In Section \ref{sec:MixedGT}, we extend the results of Section \ref{sec:Algebraic} to graphs with signed edges and directed edges.  This completes the characterization of normality for all quadratic-monomial generated domains.  Finally, we conclude in Section \ref{sec:conclusion}.

\section{Background and Notation}
\label{sec:BackAndNote}
In this section, we recall notation, definitions, and results from graph theory, semigroup theory, and edge ring theory for use in this paper.  Our notation for edge rings follows the notation of Hibi and Ohsugi \cite{OhsugiHibi:1998}.

\subsection{Graph Theory}
\label{sec:GT}

One main object of study in this paper is a signed graph.
We give the basic definitions for a signed graph in this section.

\begin{definition}\label{def:signedGraph}
A {\em signed graph} $(G,\sgn)$ is an undirected graph $G=(V,E)$ and a {\em sign function} $\sgn:E\rightarrow \{-1,+1\}$ where $\sgn(e)$ denotes the {\em sign} of the edge $e\in E$.  For notational convenience, an edge $ij$ with $\sgn(ij)=+1$ is denoted $+ij$, and an edge $ij$ with $\sgn(ij)=-1$ is denoted $-ij$.
We omit the sign when it is understood from context.
\end{definition}

\begin{definition}\label{def:components}
Let $G$ be a signed graph.
$H$ is a {\em component} of $G$ if it is a maximal connected subgraph.
We say a component $H$ is a {\em bipartite component} if the vertices of $H$ can be partitioned into two sets $L$ and $R$ so that all the edges of $H$ have exactly one vertex in $L$ and one vertex in $R$.
$G$ is {\em bipartite} if each component is a bipartite component.
\end{definition}

The construction of a ring from a graph proceeds by first constructing a semigroup from the graph.
We now define the map used in this construction.

\begin{definition}[cf {\cite{OhsugiHibi:1998}}]\label{def:edgepolytope}
Let $G$ be a signed graph with $d$ vertices, possibly with loops, and without multiple edges.
Define a map $\rho:E(G)\rightarrow \R^d$ as $\rho(e)=\sgn(e)(e_i + e_j)$ where $e=+ij$ or $e=-ij$ is an edge of the graph.
When $e$ is a loop, $i=j$ and $\rho(ii)=2\sgn(ii)e_i$.
Let $\rho(E(G))$ be the image of $E(G)$ and define the {\em edge polytope of $G$} as $\mathcal{P}_G\coloneqq\conv\{\rho(E(G))\}$.
\end{definition}

Observe that in \cite{OhsugiHibi:1998}, the authors do not consider signed graphs, and, hence, all edges $e$ in $G$ have positive sign, i.e., $\rho(e) = e_i + e_j$.

\subsection{Semigroups}\label{sec:SemiGPs}

The rings considered in this paper are constructed as semigroup rings.  In this section, we recall the definition and pertinent properties of affine semigroups, see, for example, \cite{CoxLittleSchenck-ToricVarieties,BrunsHerzog-CohenMacaulay} for more details.

\begin{definition}\label{def:affineSemigp}
An {\em affine semigroup} $C$ is a finitely generated semigroup containing zero which, for some $n$, is isomorphic to a subsemigroup of $\Z^n$.
Over a field $k$, the {\em affine semigroup ring} $k[C]$ of $C$ is the ring generated by the elements of $\{x_1^{c_1}\dots x_n^{c_n}:(c_1,\dots,c_n)\in C\}$.  We often use multi-index notation for clarity, where $x^c=x_1^{c_1}\dots x_n^{c_n}$ with $c=(c_1,\dots,c_n)$.
\end{definition}

For the remainder of this section, we assume that $C$ is an affine subsemigroup of $\Z^n$.
Let $\Z C$ denote the smallest subgroup of $\Z^n$ that contains $C$, and denote $\R_+ C\coloneqq\R_+\otimes_{\Z_+}C$ that is, the elements of $\R_+ C$ consist of a finite positive linear combinations of elements of $C$.

\begin{definition}\label{def:normalSemigp}
An affine semigroup $C$ is {\em normal} if it satisfies the following condition:
if $mz\in C$ for some $z\in \Z C$ and $m$ a positive integer, then $z\in C$.
\end{definition}

Normality can be interpreted geometrically as follows: suppose that $mz\in C$.  If the semigroup is normal, then every point of $\Z^n$ that lies on the line segment between the origin and $mz$ is either in both $C$ and $\Z C$ or neither.  In fact, normality can be determined by studying the interaction between $\R_+ C$ and $\Z C$.  Gordan's lemma gives characterizations of and connections between the normality of an affine semigroup and its semigroup ring.

\begin{proposition}[Gordan's Lemma {\cite[Proposition~6.1.2 and Theorem~6.1.4]{BrunsHerzog-CohenMacaulay}}]\leavevmode\label{prop:Gordan's}
\begin{enumerate}
\item[a)] If $C$ is a normal semigroup, then $C=\Z C\cap \R_+ C$.
\item[b)] Let $G$ be a finitely generated subgroup of $\Q^n$ and $D$ a finitely generated rational polyhedral cone in $\R^n$.
Then, $G\cap D$ is a normal semigroup.
\item[c)] Let $C$ be an affine semigroup and $k$ a field. $C$ is a normal semigroup if and only if $k[C]$ is a normal domain.
\end{enumerate}
\end{proposition}

\subsection{Integral Closures of Edge Rings}
\label{sec:IntClosures}

We recall the constructions of edge rings and their normalizations from \cite{OhsugiHibi:1998} and \cite{Villarreal:1998}, with slight generalizations.  Suppose that $R\subseteq k[x_1^{\pm 1},\dots,x_n^{\pm 1}]$ is a quadratic-monomial generated domain with generators\footnote{The authors of \cite{OhsugiHibi:1998} and \cite{Villarreal:1998} only consider generators with $s=1$.} of the form $(x_ix_j)^s$ with $s=\pm 1$.

From $R$, we construct the signed graph $G$ with vertex set $\{1,\dots,n\}$ and edges $s\cdot ij$ for each monomial of the form $(x_ix_j)^s$ in $R$.  Moreover, we construct the edge polytope $\PG$ as in Definition \ref{def:edgepolytope}.  From the edge polytope, we consider the following two affine semigroups:

\begin{definition}\label{def:T1T2}
Let $G$ be a signed graph, possibly with loops.  Define the {\em edge lattice} as the subgroup of $\Z^n$ generated by $\rho(E(G))$, i.e., $\mathcal{L}_G \coloneqq \Z \rho\{E(G)\} =\left\{\sum_{e\in G} z_e\rho(e)\ :\ z\in \Z\right\}$
There are two semigroups associated to $\PG$.
\begin{itemize}
\item $T_1\coloneqq\mbox{cone}(\mathcal{P}_G)\cap \mathcal{L}_G = \R_{+}\rho(E(G))\cap \Z\rho(E(G))$,
\item $T_2\coloneqq\Z_{+}\rho(E(G))$, i.e., the semigroup generated by $\rho(E(G))$.
\end{itemize}
\end{definition}

Each of these semigroups generates a semigroup ring of interest as follows:

\begin{definition}\label{def:Ehrhart}
Let $G$ be a signed graph, $T_1$ and $T_2$ as defined above, and let $\mathcal{A}(\PG)_n$ be the vector space over $k$ which is spanned by the monomials of the form $x^a$ such that $a\in n\mathcal{P}\cap \Z^d$.  We define the following subrings of $k[x_1^{\pm 1},\dots,x_d^{\pm 1}]$:
\begin{itemize}
\item the {\em Ehrhart} Polynomial ring: \[k[T_1] \coloneqq \APG=\bigoplus_{n=0}^{\infty} \mathcal{A}(\PG)_n\text{ and}\]
\item the {\em Edge ring of $G$}: \[k[T_2] = k[G] \coloneqq \langle x^a : a\in \rho(E(G)) \rangle.\]
\end{itemize}
\end{definition}

We observe that the ring $R$ can be recovered from the graph $G$ as $R=k[T_2]$.  Even though several signed graphs $G$ can generate the same ring $R$, this construction from $G$ motivates the notation $k[G]$ for the edge ring.  By construction, $\APG$ is a normal ring, and, in \cite{OhsugiHibi:1998} and \cite{Villarreal:1998}, the authors prove that when all edges have positive sign, $\APG$ is the normalization of $k[G]$.  Moreover, they explicitly describe the generators of $\APG$ as an algebra over $k[G]$ in terms of the structure of the graph $G$.

In the remainder of this paper, we extend these results in a more general setting.  In particular, we extend these results to rings where $s$ can be both $+1$ and $-1$ in Section \ref{sec:Algebraic}.  Then, we extend the results to arbitrary quadratic-monomial generated domains in Section \ref{sec:MixedGT}.

\begin{notation}\label{notation}
Throughout the remainder of this paper, ring notation or semigroup notation is used depending on which is clearer in the particular context.  We recall that the coordinates in $\Z^n$ correspond to the exponents of the $x_i$'s in the edge ring.
\end{notation}

\section{Algebraic Properties}
\label{sec:Algebraic}
In this section, we classify the normality of $k[G]$ in terms of combinatorial properties of the graph $G$.  Additionally, when $k[G]$ is not normal, we use the structure of $G$ to determine the additional elements needed to normalize $k[G]$.  This combinatorial condition is a generalization of the odd cycle condition of \cite{OhsugiHibi:1998}, but is more nuanced due to the more general class of signed graphs.  Our goal throughout this section is to use the combinatorial properties of $G$ to relate $k[G]$ to $\APG$.  In particular, we compute generators of $\APG$ over $k[G]$ in terms of properties of the graph.

\subsection{Odd Cycle Condition}
\label{sec:nonNorm}
In this section, we define the main combinatorial property, called the odd cycle condition, for signed graphs.  Moreover, we show that if a signed graph $G$ does not satisfy this property, then $k[G]$ is not normal.

\begin{example}\label{ex:nonnormal}
Let $G$ be the signed graph in Figure \ref{fig:examplenotNormal}, i.e., $G$ has vertex set $\{1,2,3\}$ and edge set $\{+11,+12,+23,-33\}$.
The monomial $x_1x_3^{-1}$ is in $\APG$, but it is not in $k[G]$.
We demonstrate this by showing $(1,0,-1)\in T_1\setminus T_2$ since the elements of the $T_i$'s represent the exponents of the variables.

To show that $(1,0,-1)\in T_1$ we give two ways of writing $(1,0,-1)$: one in $\cone(\PG)$ and the other in $\mathcal{L}_G$.  To show that $(1,0,-1)\in\cone(\PG)$, we write $(1,0,-1)$ as a nonnegative real combination of the edges of $G$:
\[(1,0,-1)=\frac{1}{2}\rho(+11)+\frac{1}{2}\rho(-33)=\frac{1}{2}(2,0,0)+\frac{1}{2}(0,0,-2).\]
Similarly, to show that $(1,0,-1)\in\mathcal{L}_G$, we write $(1,0,-1)$ as an integral combination of the edges of $G$:
\[(1,0,-1)=\rho(+12) - \rho(+23)  = (1,1,0) - (0,1,1).\]
Thus, $(1,0,-1)\in T_1$ and hence $x_1x_3^{-1}\in \APG$.

We now argue that $(1,0,-1)$ is not in $T_2$ because it cannot be written as a nonnegative integral combination of the edges of $G$.  In particular, if $(1,0,-1)$ were in $T_2$, then we could write $(1,0,-1) = \sum_{e\in G} z_e \rho(e)$ for $z_e\in \Z_{\geq 0}$.  Since the only two edges incident to vertex $1$ are $+11$ and $+12$, the coefficient of one of these must be at least one.  However, the vectors $(2,0,0)$ and $(1,1,0)$ contain coordinates that are strictly larger than the corresponding coordinates in $(1,0,-1)$ and there is no way to cancel these coordinates.  Therefore, $(1,0,-1)$ cannot be written as a nonnegative integral combination of the edges of $G$, $(1,0,-1)\not\in T_2$, and $x_1x_3^{-1}\notin k[G]$.

Observe that $(x_1x_3^{-1})^2 = (x_1^2)(x_3^{-1})^2\in k[G]$, and, hence, the monomial $z^2 - (x_1x_3^{-1})^2$ has the rational solution $x_1x_3^{-1}=\frac{x_1x_2}{x_2x_3}$.
Since $x_1x_3^{-1}$ is not in $k[G]$, is in the fraction field of $k[G]$, and is a solution to a monic polynomial with coefficients in $k[G]$, $k[G]$ is not normal.
\end{example}

\begin{figure}[hbt]
	\centering
		\begin{tikzpicture}
		\begin{scope}
		\node [draw,circle] (A) at (0,1) {1};
		\node [draw,circle] (B) at (1.5,1) {2};
		\node [draw,circle] (C) at (3,1) {3};
		
		\draw[line width=1pt] (A) edge node[above]{$+$} (B);
		\draw[line width=1pt] (B) edge node[above]{$+$} (C);
		\draw[line width=1pt] (A) edge[out=45,in=135,looseness=8] node[above]{$+$} (A);
		\draw[line width=1pt] (C) edge[out=45,in=135,looseness=8] node[above]{$-$} (C);
		
		\node at (1.5,0) {G};
		
		\end{scope}
		
		\begin{scope}[shift={(7,0))}]
		\node [draw,circle] (A) at (0,1) {1};
		\node [draw,circle] (B) at (1.5,1) {2};
		\node [draw,circle] (C) at (3,1) {3};
		
		\draw[line width=1pt] (A) edge node[above]{$+$} (B);
		\draw[line width=1pt] (B) edge node[above]{$-$} (C);
		\draw[line width=1pt] (A) edge[out=45,in=135,looseness=8] node[above]{$+$} (A);
		\draw[line width=1pt] (C) edge[out=45,in=135,looseness=8] node[above]{$-$} (C);
		
		\node at (1.5,0) {H};
		
		\end{scope}

		\end{tikzpicture}
		\caption[Example of a normal, and a non-normal edge ring of a signed graph.]{The edge ring for the graph $G$ is not normal.  In particular, $(x_1x_3^{-1})^2$ is in $k[G]$, but its square root is not in $k[G]$.  By changing the sign on the edge $23$, the edge ring for the graph $H$ is normal.  In this case, $x_1x_3^{-1}$ can be formed by observing that since the pair of edges $+12$ and $-23$ differ in signs, the corresponding powers on $x_2$ cancel in the product.
		 \label{fig:examplenotNormal}}
\end{figure}
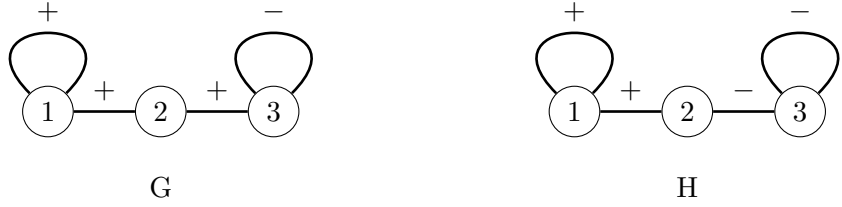

\begin{example}\label{ex:normal}
Let $H$ be the signed graph in Figure \ref{fig:examplenotNormal}, i.e., $H$ has vertex set $\{1,2,3\}$ and edge set $\{+11,+12,-23,-33\}$.
Observe that $H$ is the same graph as $G$ from Example \ref{ex:nonnormal}, except that edge $+23$ has been replaced by $-23$.  In this case, we can see that $(1,0,-1)\in T_1$ by a similar set of coefficients as in Example \ref{ex:nonnormal}.  On the other hand, $(1,0,-1)\in T_2$ since $(1,0,-1)$ can be written as a nonnegative integral combination of the edges.
\[(1,0,-1)=\rho(+12)+\rho(-23).\]
Therefore, $x_1x_3^{-1}$ is not a witness to non-normality.  In fact, $k[H]$ is normal.
\end{example}

Examples \ref{ex:nonnormal} and \ref{ex:normal} illustrate a subtlety in detecting normality in the case of a signed graph since the normality depends not only on the structure of the graph, but also the signs of the edges.  These examples lead to the odd cycle condition for signed graphs:

\begin{definition}\label{def:OCC}
Let $G$ be a signed graph, we say that $G$ satisfies the {\em odd cycle condition} if for every pair of odd cycles $C$ and $C'$ of $G$, at least one of the following occurs: 
\begin{enumerate}
\item $C$ and $C'$ are in distinct components,
\item $C$ and $C'$ have a vertex in common,
\item $C$ and $C'$ have a sign alternating $i,j$-path in $G$, where $i$ is a vertex of $C$ and $j$ is a vertex of $C'$.
\end{enumerate}
\end{definition}

Note that the definition of the odd cycle condition extends the definition from \cite{OhsugiHibi:1998} and \cite{Villarreal:1998}.  In particular, since the graphs considered by Hibi and Ohsugi are connected and they have only positive edges, the only possible alternating paths are of length one.  By restricting to that case, the odd cycle condition of \cite{OhsugiHibi:1998} follows.  For connected graphs with all positive edges, \cite{OhsugiHibi:1998} shows that the odd cycle condition is equivalent to normality of the edge ring $k[G]$.  We now begin the generalization of this proof to the case of a disconnected signed graph by showing that the odd cycle condition is a necessary condition on $G$ for $k[G]$ to be normal.

\begin{definition}\label{def:signature}
Let $G$ be a signed graph, $C$ be a cycle in $G$, and $\{u_1,\dots,u_n\}$ be the vertices of $C$.
We define the {\em signature} of vertex $u_i$ in $C$ as the average of the signs of the edges incident to $i$ in $C$, i.e., $\sig_C(u_i)=\frac{1}{2}(\sgn(u_{i-1}u_i)+\sgn(u_iu_{i+1}))$ where the subscripts are computed modulo $n$.  Observe that the signature of a vertex is one of $1$, $0$, or $-1$.
\end{definition} 

\begin{observation}\label{obs:alternating}
Let $G$ be a signed graph and suppose that $W$ is an alternating walk with vertices $\{i_0,\dots,i_n\}$ and edges $\{i_0i_1,\dots,i_{n-1}i_n\}$.  Then,
\[
\prod_{j=1}^n(x_{i_{j-1}}x_{i_j})^{\sgn(i_{j-1}i_j)}=x_{i_0}^{\sgn(i_0i_1)}x_{i_n}^{\sgn(i_{n-1}i_n)}.
\]
The intermediate variables cancel because each $x_{i_j}$, with $j\not\in\{0,n\}$, appears in two factors with opposite signs.  Moreover, if $G$ contains a cycle $C$, $W\subseteq C$, and $i_0$ and $i_n$ have nonzero signature in $C$, then, $\sgn(i_0i_1)=\sig_C(i_0)$ and $\sgn(i_{n-1}i_n)=\sig_C(i_n)$, so, in this case, the product equals $x_0^{\sig_C(i_0)}x_n^{\sig_C(i_n)}$.
\end{observation}

Throughout the remainder of this section, we fix a signed graph $G$ which does not satisfy the odd cycle condition.  Since $G$ does not satisfy the odd cycle condition, $G$ contains cycles $C$ and $C'$ that violate the odd cycle condition, i.e., $C$ and $C$' are disjoint, are in the same component, and do not have an alternating path between them.  We prove that the following monomial is a witness to the fact that $k[G]$ is not normal.
\begin{equation}\label{eq:goalproduct}
\prod_{\ell\in C}x_{\ell}^{\sig_C(\ell)}\prod_{\ell\in C'}x_{\ell}^{\sig_{C'}(\ell)}
\end{equation} 

\begin{lemma}\label{lem:converse:Ehrhart}
Let $G$ be a signed graph, which does not satisfy the odd cycle condition, then Expression (\ref{eq:goalproduct}) is in $\APG$.
\end{lemma}
\begin{proof}
To prove this statement, we write the monomial in Expression (\ref{eq:goalproduct}) as both a positive real combination and an integral combination of the images of the edges of $E$ under $\rho$.

First, observe that this product can be constructed by putting the weight of $\frac{1}{2}$ on all the edges of $C\cup C'$.  In particular,
\begin{equation}\label{eq:APG}
\prod_{\pm k\ell\in C\cup C'}(x_kx_\ell)^{\frac{1}{2}\sgn(k\ell)}=\prod_{\ell\in C}x_{\ell}^{\sig_C(\ell)}\prod_{\ell\in C'}x_{\ell}^{\sig_{C'}(\ell)}.
\end{equation}
Since all the exponents (weights) are positive real numbers, Expression (\ref{eq:goalproduct}) is in $k[\cone(\PG)]$.

Since $C$ and $C'$ are odd cycles, by a parity argument, there exists vertices $i$ in $C$ and $j$ in $C'$ with nonzero signature.
Since $C$ and $C'$ are in the same component, let $P$ be any path between $i$ and $j$.  We choose a walk $W$ in $G$ between $i$ and $j$ with the following property: if $\sig_C(i)=\sig_{C'}(j)$, then $W$ has odd length, and if $\sig_C(i)=-\sig_{C'}(j)$, then $W$ has even length.
Such a walk can be constructed by considering $W=P$ or $W=P\cup C$, depending on the parity of $P$.

Suppose that the vertices of $W$ are $\{i=i_0,\dots,i_n=j\}$ and the edges of $W$ are $\{i_0i_1,i_1i_2,\dots$ $,i_{n-1}i_n\}$ (possibly with repeated vertices or edges).
Let $a_1,\dots,a_n$ be $\pm 1$ so that $a_\ell\sgn(i_{\ell-i}i_\ell)=(-1)^{\ell-1}\sig_C(i)$ for $1\leq \ell\leq n$.
In other words, $a_1\sgn(i_0i_1)=\sig_C(i)$, $a_2\sgn(i_1i_2)=-\sig_C(i)$, and, due to choice of the length of the walk, $a_{n}\sgn(i_{n-1}i_n)=(-1)^n\sig_{C}(i)=\sig_{C'}(j)$. 
Therefore, while the signs on $W$ may not alternate, the products $a_\ell\sgn(i_{\ell-1}i_\ell)$ alternate.  Therefore, in the following product, the intermediate terms cancel by Observation \ref{obs:alternating}:
\begin{equation}\label{eq:connectingproduct}
\prod_{\ell=1}^{n} (x_{i_{\ell-1}}x_{i_{\ell}})^{a_\ell\sgn_C(i_{l-1}i_l)}=x_i^{\sig_C(i)}x_j^{\sig_{C'}(j)}.
\end{equation}

Consider $C\setminus i$, which is a path with an even number of vertices.  Let $\{u_1,\dots,u_{2m}\}$ be the vertices of $C\setminus i$, in order.  We now construct the product
\begin{equation}\label{eq:productwithouti}
\prod_{\ell\in C\setminus i}x_{\ell}^{\sig_C(\ell)}.
\end{equation}
Observe that, since the signature is a weighted sum of the signed edges, we can rearrange a sum of signatures into a sum over edges as follows:
\begin{equation}\label{eq:rewrigingsumP1}
\sum_{\ell\in C} \sig_C(\ell) = \sum_{\ell \in C}\frac{1}{2}\sum_{\substack{j\\\pm j\ell\in C}}\sgn(j\ell) = \sum_{e\in C} \sgn(e).
\end{equation}
Since $\sig_C(i) = \frac{1}{2}\left(\sum_{\pm ji\in C} \sgn(ij)\right)$, we can write a similar equation for $C\setminus i$:
\begin{equation}\label{eq:rewritingsum}
\sum_{\ell\in C\setminus i}\sig_C(\ell)=\frac{1}{2}\sum_{\substack{j\\\pm ij\in C}} \sgn(ij)+\sum_{e\in C\setminus i} \sgn(e)
=\sig_C(i)+\sum_{k=1}^{2m-1}\sgn(u_ku_{k+1}).
\end{equation}

Equation (\ref{eq:rewritingsum}) may be justified as follows: every edge $e\in C$ (except those incident to $i$) appears twice with weight $\frac{1}{2}$ in the LHS of Equation (\ref{eq:rewritingsum}).  Additionally, the edges incident to $i$ appear once, with weight $\frac{1}{2}$ from the signature of each of the neighbors of $i$.

Since the right-hand-side of Equation (\ref{eq:rewritingsum}) contains an even number of terms which are all equal to $\pm 1$, we know that both sides of the equation are even.
Hence, using a parity argument, we see that there are an even number of vertices in $C\setminus i$ with nonzero signature.
Let $\{u_{j_1},\dots,u_{j_{2p}}\}$ be the vertices of $C\setminus i$ with nonzero signature, in order around $C$.
Then, there are disjoint alternating paths $\{W_1,\dots,W_p\}$ such that $W_k$ ends at $u_{j_{2k-1}}$ and $u_{j_{2k}}$; these paths are alternating because otherwise there would be a vertex with nonzero signature between $u_{j_{2k-1}}$ and $u_{j_{2k}}$.  For each $W_k$, we assign an edge weight of $1$ on the edges of $W_k$ so that the corresponding ring element is $x_{j_{2k-1}}^{\sig_C(u_{j_{2k-1}})}x_{j_{2k}}^{\sig_C(u_{j_{2k}})}$ by Observation \ref{obs:alternating}.  By multiplying these products over all paths, the product is Expression (\ref{eq:productwithouti}).  Similarly, we can construct $\prod_{\ell\in C'\setminus j}x_{\ell}^{\sig_{C'}(\ell)}$.  By multiplying Equation (\ref{eq:connectingproduct}) with Expression (\ref{eq:productwithouti}) as well as the corresponding expression for $C'\setminus j$, the result is 
\begin{equation*}%\label{eq:nonnormalproduct}
\left(x_i^{\sig_C(i)}x_j^{\sig_{C'}(j)}\right)\prod_{\ell\in C\setminus i}x_{\ell}^{\sig_C(\ell)}\prod_{\ell\in C'\setminus j}x_{\ell}^{\sig_{C'}(\ell)}=\prod_{\ell\in C}x_{\ell}^{\sig_C(\ell)}\prod_{\ell\in C'}x_{\ell}^{\sig_{C'}(\ell)}.
\end{equation*}
Therefore, the exponents in Expression (\ref{eq:goalproduct}) can be written as an integral combination of $\rho(e)$ for $e\in G$, i.e., Expression (\ref{eq:goalproduct}) is in $k[\mathcal{L}_G]$.  Thus, Expression (\ref{eq:goalproduct}) is in $\APG$.
\end{proof}

\begin{lemma}\label{lem:converse:Edge}
Let $G$ be a signed graph, which does not satisfy the odd cycle condition, then Expression (\ref{eq:goalproduct}) is not in $k[G]$
\end{lemma}
\begin{proof}
Suppose, for contradiction, that there are nonnegative integer weights on the edges of $G$ whose corresponding ring element is Expression (\ref{eq:goalproduct}).

Let $\mathcal{E}$ be the multiset of the edges of $G$ with multiplicity according to their weights.  All vertices other than those in $C$ or $C'$ with nonzero signature do not appear in Expression (\ref{eq:goalproduct}).  Therefore, for any such vertex $v$, the number of positive or negative edges in $\mathcal{E}$ incident to $v$ must be equal.  Thus, we can partition $\mathcal{E}$ into a collection of closed alternating walks and alternating walks whose endpoints are points in $C$ or $C'$ with nonzero signature.  With this construction, there are four types of walks:
\begin{inparaenum}[(1)]
\item closed walks, 
\item walks with both endpoints in $C$,
\item walks with both endpoints in $C'$, and 
\item walks with one endpoint in $C$ and the other in $C'$.
\end{inparaenum}
Since $C$ and $C'$ violate the odd cycle condition, there are no walks of Type (4).
Similarly, Type (1) closed walks correspond to the identity in the ring and can be discarded, see Theorem \ref{thm:equivalence}.  Therefore, we reduce to the case where $\mathcal{E}$ can be partitioned into walks of Types (2) and (3).  By Observation \ref{obs:alternating}, for any walk of Type (2), with endpoints $i_0$, and $i_n$, the corresponding ring element is $x_{i_0}^{\sgn(i_0i_1)}x_{i_n}^{\sgn(i_{n-1}i_n)}$; a similar statement holds in $C'$ for walks of Type (3).  We established in Lemma \ref{lem:converse:Ehrhart} that $C$ and $C'$ each have an odd number of vertices with nonzero signature.  However, since each walk of Type (2) introduces a monomial of even total degree, it is impossible for the product $\prod_{\ell\in C}x_{\ell}^{\sig_C(\ell)}$, which is of odd total degree, to be the product of walks of Type (2).
\end{proof}

\begin{theorem}\label{thm:converse}
Let $G$ be a signed graph which does not satisfy the odd cycle condition, then $k[G]$ is not normal.
\end{theorem}
\begin{proof}
Since Expression (\ref{eq:goalproduct}) is in $k[\mathcal{L}_G]$, by Lemma \ref{lem:converse:Ehrhart}, the monomial is in the fraction field of $k[G]$.  Moreover, since the square of Equation (\ref{eq:APG}) has all nonnegative integral weights, the square of Expression (\ref{eq:goalproduct}) is in $k[G]$.  Therefore, since Expression (\ref{eq:goalproduct}) is the root of a monic polynomial in $k[G]$, is in the fraction field of $k[G]$, and is not in $k[G]$, by Lemma \ref{lem:converse:Edge}, it follows that $k[G]$ is not normal.
\end{proof}

In the remainder of this section, we prove that the odd cycle condition is also a sufficient condition for $k[G]$ to be normal.  To reach this result, we reduce the graphs of interest to a few special cases in the next section.  The proof of Theorem \ref{thm:converse} includes an important construction for these structural results that we collect here: 

\begin{observation}\label{obs:productconstruction}
If $C$ is an odd cycle in a graph $G$ and $i$ is a vertex of $C$ such that $\sig_C(i)\not=0$, then there are edge weights of $0$ and $1$ on $C$ whose product is 
\[
\prod_{\ell\in C\setminus i}x_{\ell}^{\sig_C(\ell)}.
\]
\end{observation}

\subsection{Reductions in \texorpdfstring{$T_1$}{T1}}\label{sec:AlterForms}

In this section, we explicitly translate between certain combinatorial graph properties and ring properties.  We then use this translation to reduce the graphs under consideration to simpler forms, which are easier to study.  In particular, we provide the combinatorial property such that the corresponding product in $\APG$ is $1$.  We then use this combinatorial property to reduce elements of $\APG$.  This is the main link we use to determine normality in the remainder of this section.

\begin{theorem}\label{thm:equivalence}
Let $G$ be a signed graph, and, for each edge $ij$ in $G$, let $a_{ij}$ be an integer.
The following are equivalent:
\begin{enumerate}
\item $\prod_{\pm ij\in E} (x_ix_j)^{\sgn(ij)a_{ij}}=1$.\label{cond:firstcondition:equivalence}
\item For all vertices $i$, $\sum_{j\sim i} \sgn(ij)a_{ij} = 0$ where the sum is taken over all $j$ adjacent to $i$.\label{cond:secondcondition:equivalence}
\item There is a multiset of closed walks allowing loops with weights $w_{ij}=\pm 1$ so that the products $w_{ij}\sgn(ij)$ alternate along the closed walks and $a_{ij}$ is the sum of the weights of the occurrences of edge $ij$ in the closed walks.\label{cond:fourthcondition:equivalence}
\end{enumerate}
\end{theorem}

\begin{proof}
Since the exponent of $x_i$ in $1$ is the sum in $2$, the equivalence of Conditions $1$ and $2$ is straight forward.

We now prove the equivalence of $2$ and $3$.  If a collection of closed walks satisfying Condition \ref{cond:fourthcondition:equivalence} exist, then, let $ji$ and $ik$ be consecutive edges along one such closed walk.  Since the weighted walks are alternating, $w_{ji}\sgn(ji)+w_{jk}\sgn(jk)=0$.  Therefore, for a fixed vertex $i$, the weighted sum over all incident edges is $0$ because the edges come in alternating pairs.  By reorganizing this sum to combine repeated edges into the $a_{ij}$'s, the sum over all edges incident to $i$ is precisely the sum $\sum_{j\sim i} \sgn(ij)a_{ij}$, and is zero.  Since $i$ is arbitrary, the first statement follows.

On the other hand, suppose that the sums in Condition \ref{cond:secondcondition:equivalence} are zero for all vertices.   Let $\mathcal{E}$ be the multiset of pairs of signed edges of $G$ and $\pm 1$.  In particular, for each edge $jk\in G$ with $j\not=k$, the pair $(jk,\sgn(a_{jk}))$ occurs $|a_{jk}|$ times, and, for each loop, the pair $(jj,\sgn(a_{jj}))$ occurs $|2a_{jj}|$ times.  We observe that for fixed $i$, by assumption, the number of pairs $(ij,+1)$ and $(ik,-1)$ in $\mathcal{E}$ are equal where $j$ and $k$ vary over the neighbors of $i$.  For each $i$, choose a perfect matching between $(ij,+1)$ and $(ik,-1)$, where $j$ or $k$ may equal $i$.  The walks formed by following the matched edges form a collection of alternating closed walks since $\mathcal{E}$ is finite.
\end{proof}

Theorem \ref{thm:equivalence} is used to simplify the graph under consideration.  We now explicitly describe the condition that we use in our reductions.

\begin{definition}\label{def:reducingClosedWalk}
Let $G$ be a signed graph and $W$ a closed walk in $G$.  $W$ is a {\em reducing closed walk} if $W$ has even length and if $e\in G$, appears in $W$ multiple times, the number of edges between occurrences of $e$ is odd.  
\end{definition}

\begin{corollary}\label{cor:equivalence}
Let $G$ be a signed graph and $W$ a reducing closed walk of $G$.  Then there are nonzero edge weights $a_{ij}$ on the edges of $W$ satisfying the equivalence of Theorem \ref{thm:equivalence}.
\end{corollary}
\begin{proof}
We assign weights $w_e=\pm 1$ to the edges of $W$ so that $w_e\sgn(e)$ is alternating around $W$.  For an edge $e$ that occurs multiple times in $W$, since the number of edges between occurrences of $e$ is odd, $w_e$ has the same sign for every occurrence.  For each edge $e$, let $a_e$ be the number of occurrences of $e$ in $W$ times $w_e$, in other words, the sum of the weights for that edge for the multiple occurrences in $W$.  We observe that $a_e$ is never zero for edges in $W$ since the weights all have the same sign and cannot cancel.  Since $w_e\sgn(e)$ is alternating in $W$, the sum in Condition \ref{cond:secondcondition:equivalence} is a collection of pairs which sum to zero, so the conditions of Theorem \ref{thm:equivalence} hold.
\end{proof}

We now use Corollary \ref{cor:equivalence} to reduce the graph under consideration by rewriting the elements of $T_1$ in alternate ways.  This reduction allows us to reduce our attention to trees and unicyclic graphs with odd cycles.  In Section \ref{sec:sufficiency}, these reductions are used to complete the proof that the odd cycle condition is equivalent to the normality of $k[G]$.  

\begin{lemma}\label{lem:reducing}
Let $G$ be a signed graph, and assume that, for each edge $e$ of $G$, the weight $a_e$ is positive.  Suppose that $G$ contains a reducing closed walk $W$.  Then there is a proper subgraph $H$ of $G$ with fewer edges and positive edge weights $b_e$ so that $\sum_{e\in G} a_e \rho(e)=\sum_{e\in H} b_e\rho(e)$.
\end{lemma}
\begin{proof}
Let $c_e$ be the weights on $W$ constructed in Corollary \ref{cor:equivalence}.  Let $e'$ be the edge of $W$ so that $\frac{a_{e'}}{c_{e'}}$ is minimized (which is negative).  For each edge $e\in W$, we add $\left|\frac{a_{e'}}{c_{e'}}\right|c_e$ to its weight.  In other words, the weights on the edges $e\in G\setminus W$ are unchanged, i.e., $b_e=a_e$, and the weights on the edges $e\in W$ are $b_e=a_e+\left|\frac{a_{e'}}{c_{e'}}\right|c_e$.  By the minimality of $e'$, all the weights $b_e$ are nonnegative, and edge $e'$ has weight $0$ (since $c_{e'}$ is negative).  Let $H$ be the subgraph of $G$ whose edges have nonzero weights.  By Corollary \ref{cor:equivalence}, we know that $\sum_{e\in W}c_e\rho(e)=0$.  Therefore, since the sums differ by a multiple of $0$, the desired equality holds.
\end{proof}

We now use this lemma to reduce any graph to one of a simpler form, which is easier to study.

\begin{corollary}\label{cor:reduction1}
Let $G$ be a signed graph, and assume that for each edge $e$ of $G$, the weight $a_e$ is positive.  Then, there is a subgraph $H$ where each component of $H$ is a tree or a unicyclic graph with an odd cycle and there are positive edge weights $b_e$ on $H$ so that $\sum_{e\in G} a_e \rho(e)= \sum_{e\in H} b_e\rho(e)$.
\end{corollary}
\begin{proof}
Iteratively applying the reduction in Lemma \ref{lem:reducing}, we assume that $H$ is a subgraph of $G$ which does not contain a reducing closed walk.  We observe that $H$ cannot contain an even cycle because an even cycle is a reducing closed walk.  Moreover, if $H$ contains two odd cycles $C$ and $C'$ which are not disjoint, then their union contains an even cycle, which is a reducing closed walk.  Finally, if $H$ contains two disjoint odd cycles $C$ and $C'$ in the same component, let $P$ be a path between them whose interior is disjoint from $C$ and $C'$.  Let $W$ be the walk formed by walking around $C$ and $C'$ once and path $P$ twice.  $W$ has even length since $C$ and $C'$ both have odd length and $P$ is covered twice.  Moreover, the only edges which occur multiple times are those in $P$, but the edges between a repeated edge consist of a portion of $P$ covered twice and one of the cycles, i.e., an odd number of edges.  Therefore, $W$ is a reducing closed walk.  The only possible cases are those in the statement of the corollary.
\end{proof}

\subsection{Sufficiency of Odd Cycle Condition}\label{sec:sufficiency}

In this section, we use Corollary \ref{cor:reduction1} to show that the odd cycle condition for the graph $G$ is equivalent to the normality of the ring $k[G]$.  We begin with two lemmas that give key structural results in the main theorem.

\begin{lemma}\label{lem:oddcyclesonly}
Let $G$ be a signed graph and $\alpha\in T_1$.  Let $H$ be a subgraph of $G$ such that each component of $H$ is a tree or a unicyclic graph with an odd cycle.  Suppose that $\alpha=\sum_{e\in H}a_e\rho(e)$ where $0<a_e<1$.  Then, $H$ is a collection of odd cycles and $a_e=\frac{1}{2}$ for all $e\in H$.
\end{lemma}
\begin{proof}
We recall that since $\alpha\in T_1$, $\alpha$ is an integral sum of $\rho(e)$ for $e\in G$.  Therefore, for any vertex $i$, $\sum_{j\sim i}\sgn(ij)a_{ij}$ is an integer, where the sum is taken over all $j$ adjacent to $i$.

Suppose, for contradiction, that $i\in H$ is a leaf.  Since $i$ is a leaf, the sum $\sum_{j\sim i}\sgn(ij)a_{ij}$ consists of a single term, which is not an integer, a contradiction.  Therefore, $H$ is a union of disjoint odd cycles.  

Let $C$ be a cycle in $H$.  If $C$ is a loop at $i$, then, by the observation above, we know that $2\sgn(ii)a_{ii}$ is an integer, so $a_{ii}=\frac{1}{2}$.  If $C$ is not a loop, suppose that $k$, $i$, and $\ell$ are consecutive vertices in $C$.  Then, the sum $\sum_{j\sim i}\sgn(ij)a_{ij}$ reduces to $\sgn(ik)a_{ik}+\sgn(i\ell)a_{i\ell}$.  By our assumptions, the sum is one of $-1$, $0$, or $1$.  By a case-by-case analysis, we see that either $a_{i\ell}=a_{ik}$ or $a_{i\ell}=1-a_{ik}$.  By repeating this argument for each vertex of $C$, we see that we can partition the edges of $C$ into two classes, those with weight $a_{ik}$ and those with weight $1-a_{ik}$.  By a parity argument, there is either: (1) some vertex $s$ with neighbors $r$ and $t$ so that $a_{rs}=1-a_{st}$ and $\sgn(rs)a_{rs}+\sgn(st)a_{st}=0$ or (2) some vertex $s$ with neighbors $r$ and $t$ so that $a_{rs}=a_{st}$ and $\sgn(rs)a_{rs}+\sgn(st)a_{st}=\pm1$.  In either case, we find that $a_{rs}=\frac{1}{2}$, so all weights of $C$ are $\frac{1}{2}$.
\end{proof}

\begin{lemma}\label{lem:keyproduct}
Let $G$ be a signed graph with $C_1$ and $C_2$ disjoint odd cycles in the same component.  Suppose that $P$ is an alternating path between $C_1$ and $C_2$ whose interior is disjoint from $C_1\cup C_2$.  Then $k[G]$ contains
\[
\prod_{i\in C_1}x_i^{\sig_{C_1}(i)}\prod_{j\in C_2}x_j^{\sig_{C_2}(j)}.
\]
\end{lemma}
\begin{proof}
Suppose $P$ has vertices $\{i_0=i,\dots,i_n=j\}$ and edges $\{i_0i_1,\dots,i_{n-1}i_n\}$, where $i\in C_1$ and $j\in C_2$.  Let $P'$ be a maximal alternating walk of the form $P'=P_1PP_2$ where $P_1$ and $P_2$ are, possibly empty, walks in $C_1$ and $C_2$, respectively.  We observe that if $\sgn(i_0i_1)=\sig_C(i)$, then the path $P$ cannot be extended.  Consider the case where the equality does not hold.  In this case, at least one incident edge to $i$ in $C$ extends the alternating walk.  We continue to add edges to $P_1$ until encountering a vertex with nonzero signature.  There always exists such a vertex because, by Lemma \ref{lem:converse:Ehrhart}, there are an odd number of vertices in $C_1$ with nonzero signature.  A similar statement holds for $j\in C_2$.  

Let $P'$ have vertices $\{j_0=i',\dots,j_m=j'\}$ and edges $\{j_0j_1,\dots,j_{m-1}j_m\}$, where $i'\in C_1$ and $j'\in C_2$.  Then $P'$ is an alternating path such that $\sgn(j_0j_1)=\sig_{C_1}(i')$ and $\sgn(j_{m-1}j_m)=\sig_{C_2}(j')$.  Therefore, by assigning a weight of $1$ to all the edges in $P'$, it follows from Observation \ref{obs:alternating} that $k[G]$ contains
\begin{equation}\label{eq:Obs1:endpoints}
x_{i'}^{\sgn(j_0,j_1)}x_{j'}^{\sgn(j_{m-1}j_m)}=x_{i'}^{\sig_{C_1}(i')}x_{j'}^{\sig_{C_1}(j')}.
\end{equation}

To complete the theorem, it is enough to show that $k[G]$ contains
\begin{equation}\label{eq:Obs2:products}
\prod_{\ell\in C_1\setminus i'}x_\ell^{\sig_{C_1}(\ell)}\qquad\text{and}\qquad\prod_{\ell\in C_2\setminus j'}x_\ell^{\sig_{C_2}(\ell)}.
\end{equation}
By Observation \ref{obs:productconstruction}, there are weights of $1$ and $0$ on the edges of $C_1$ and $C_2$ to achieve this product.  Therefore, by combining Equation (\ref{eq:Obs1:endpoints}) and Expression (\ref{eq:Obs2:products}), we get the desired monomial.
\end{proof}

With these lemmas in hand, we are able to prove our a main theorem of this section and characterize the normality of edge rings.

\begin{theorem}\label{thm:MainResult}%\label{thm:reduction}
Let $G$ be a signed graph, then $k[G]$ is normal if and only if $G$ satisfies the odd cycle condition.
\end{theorem}
\begin{proof}
Theorem \ref{thm:converse} shows that a graph which is normal must satisfy the odd cycle condition.  We now prove the other direction.  Suppose that $G$ is a signed graph which satisfies the odd cycle condition, and suppose that $\alpha$ is in $T_1$.  Our goal is to show $\alpha$ is in $T_2$.  Since $\alpha\in T_1$, we can write $\alpha=\sum_{e\in G}a_e\rho(e)$ with $a_e\geq 0$.  Let $G'$ be the subgraph of $G$ consisting of the edges $e\in G$ where $a_e$ is positive.  By Corollary \ref{cor:reduction1} we can find a subgraph $H'$ whose components are trees or unicyclic graphs with odd cycles such that $\alpha=\sum_{e\in H}a'_e\rho(e)$ with $a'_e$ positive.

Let $\lfloor\alpha\rfloor=\sum_{e\in H'}\lfloor a'_e\rfloor\rho(e)$.  Since $\lfloor\alpha\rfloor$ is a nonnegative integral combination of $\rho(e)$ for $e\in G$, it follows that $\lfloor\alpha\rfloor\in T_2$.  Let $\beta=\alpha-\lfloor\alpha\rfloor$.  In other words, $\beta=\sum_{e\in H'}b_e\rho(e)$ where $b_e=a'_e-\lfloor a'_e\rfloor$.  We observe that $\beta$ can be written as both (1) a nonnegative linear combination of $\rho(e)$ for $e\in G$ and (2) a difference of two integral linear combinations of $\rho(e)$ for $e\in G$.  Therefore, $\beta\in T_1$.  Throughout the rest of the proof, we prove that $\beta\in T_2$ because then $\alpha=\beta+\lfloor\alpha\rfloor$ would be in $T_2$ as it would be a sum of elements of $T_2$.  

Let $H$ be the subgraph consisting of the edges $e\in H'$ where $b_e$ is positive.  By Lemma \ref{lem:oddcyclesonly}, we know that $H$ is a set of odd cycles and $b_e=\frac{1}{2}$.  For a cycle $C$ of $H$, let $\beta_C$ be the restriction of $\beta$ to $C$, i.e., $\beta_C=\sum_{e\in C}b_e\rho(e)$.  Then,
\[
x^{\beta_C}=\prod_{\ell\in C}x_\ell^{\sig_C(\ell)},
\]
which is the same monomial considered in Theorem \ref{thm:converse}.  Therefore, we observe that normality of $k[G]$ is completely based on these products on cycles.

Let $G_1$ be a component of the original graph $G$ and $\{C_1,\dots,C_m\}$ the cycles of $H$ contained in $G$.  We prove that the number of these cycles is even.  Let $\beta_{G_1}$ be the restriction of $\beta$ to $G_1$, i.e., $\beta_{G_1}=\sum_{e\in G_1}b_e\rho(e)$.  Observe that $x^{\beta_{G_1}}=\prod x^{\beta_{C_i}}$.  Since $\beta\in T_1$, $\beta$ can be written as an integral combination of $\rho(e)$ for $e\in G$.  Moreover, since $x^{\rho(e)}$ has even degree, $x^{\beta_{G_1}}$ has even degree.  On the other hand, by Theorem \ref{thm:converse}, we know that each $x^{\beta_{C_i}}$ has odd degree.  Therefore, by considering the parity, there must be an even number of cycles in $\{C_1,\dots,C_m\}$.  Let $C_1$ and $C_2$ be any two cycles in $G_1$.  Since $G_1$ satisfies the odd cycle condition, there is an alternating path between $C_1$ and $C_2$.  By Lemma \ref{lem:keyproduct}, $x^{\beta_{C_1}}x^{\beta_{C_2}}\in k[G]$.  Therefore, by pairing up cycles, $x^{\beta_{G_1}}=\prod x^{\beta_{C_i}}$ is in $k[G]$.  Finally, since the choice of component was arbitrary, $x^\beta=\prod x^{\beta_{G_i}}$ is in $k[G]$.  Hence, $\beta\in T_2$, so $\alpha\in T_2$.
\end{proof}

The characterization of the normality of toric varieties in Theorem \ref{thm:MainResult} can be strengthened by describing the generators $\APG$ over $k[G]$.  

\begin{definition}\label{def:signedPi}
Let $G$ be a signed graph.
A pair $\Pi=\{C,C'\}$ of disjoint odd cycles in a component of $G$ is {\em exceptional} if there does not exist an alternating path connecting $C$ and $C'$.
For an exceptional pair $\Pi=\{C,C'\}$, define \[\frac{1}{2}\rho(\Pi)=\frac{1}{2}\sum_{\pm ij\in C}\rho(ij)+\frac{1}{2}\sum_{\pm ij\in C'}\rho(ij),\] and, in $k[x_1,\dots,x_n]$, let
\[M_{\Pi}=x^{\frac{1}{2}\rho(\Pi)}=\prod_{\ell\in C}x_\ell^{\sig_C(x_\ell)}\prod_{\ell\in C'}x_\ell^{\sig_{C'}(x_\ell)}.\]
\end{definition}

Observe that for a pair of exceptional odd cycles $\Pi=\{C,C'\}$, the monomial $M_{\Pi}$ is precisely the monomial which appears in Theorems \ref{thm:converse} and \ref{thm:MainResult}.

\begin{theorem}\label{thm:Normalization}
Let $G$ be a signed graph and $k[G]$ the edge ring of $G$.  Let $\Pi_1=\{C_1,C_1'\},\dots,\Pi_q=\{C_q,C_q'\}$ denote the exceptional pairs of odd cycles in $G$.  Then, $\APG$ is generated by the monomials $M_{\Pi_1}\dots,M_{\Pi_q}$ as an algebra over $k[G]$.
\end{theorem}
\begin{proof}
By \cite[Lemma 2.7]{Villarreal:1998}, the normalization for an edge ring with several components is the tensor product of the normalizations of the edge rings of the components.  Therefore, to compute the normalization of $k[G]$, it is sufficient to assume that $G$ is connected.  In the proof of Theorem \ref{thm:MainResult}, proving that $\alpha\in T_1$ was equivalent to proving that $\beta_{C_1}+\beta_{C_2}\in T_2$ for a pair of odd cycles $C_1$ and $C_2$.  By Lemma \ref{lem:keyproduct}, if $\{C_1,C_2\}$ is not an exceptional pair, then the sum is in $T_2$.  Therefore, monomials corresponding to exceptional pairs are the only extra monomials needed.
\end{proof}

\section{Normal Domains for Mixed Signed, Directed Graphs}
\label{sec:MixedGT}

In Section \ref{sec:Algebraic}, we provided a combinatorial characterization of the normality of quadratic-monomial generated rings where the generators are of the form $x_ix_j$ or $x_i^{-1}x_j^{-1}$.  This section contains one of the main results of this paper, where we extend this characterization to all quadratic-monomial generated domains, i.e., we allow generators of the form $x_i^{-1}x_j$. 

\subsection{Reductions to Signed Graphs}\label{sec:recductions:signed}

In this section, we develop the combinatorial properties to describe arbitrary quadratic-monomial generated domains.  In particular, we show that generators of the form $x_i^{-1}x_j$ correspond, combinatorially, to directed edges and relate the edge rings of mixed signed, directed graphs to those of signed graphs.

\begin{definition}\label{def:MixedGraphs}
A {\em mixed signed, directed graph} is a pair $G=(V,E)$ of {\em vertices}, $V$, and {\em edges}, $E$, where $E$ consists of a set of signed edges and {\em directed edges} between distinct vertex pairs.
As before, we denote a positive edge between $i$ and $j$ as $+ij$ and a negative edge between $i$ and $j$ as $-ij$.
A directed edge from $i$ to $j$ is denoted $(i,j)$.
\end{definition}

Note that, for any vertex $i$, there may be positive and negative loops at $i$ denoted $+ii$ and $-ii$ respectively, but not directed loops (since directed loops correspond to the indentity).
Also, for a pair $i$ and $j$ of distinct vertices, any subset of the four possible edges $+ij,-ij,(i,j)$ and $(j,i)$ can be edges in $G$.  We now extend the definitions pertaining to edge rings to this case.

\begin{definition}\label{def:MixedPolytopes}
Let $G$ be a mixed signed, directed graph with $d$ vertices, possibly with loops, and multiple edges.
Define $\rho:E(G)\rightarrow \R^d$ as $\rho(e)=\sgn(e)(e_i + e_j) \in \R^d$ when $e=\sgn(ij)ij$ is a signed edge of the graph and as $\rho(e)=e_j-e_i$ when $e=(i,j)$ a directed edge of the graph.  
\end{definition}

Using this definition for $\rho$, the {\em edge polytope of $G$} is defined as in Definition \ref{def:edgepolytope}, the semigroups $T_1$ and $T_2$ are defined as in Definition \ref{def:T1T2}, and the Ehrhart and edge rings are defined as in Definition \ref{def:Ehrhart}.

Let $G$ be a mixed signed, directed graph with vertices $i$, $j$, and $t_{(i,j)}$.  Suppose that $(i,j)$ is a directed edge of $G$ and $-it_{(i,j)}$ and $+t_{(i,j)}j$ are signed edges of $G$.  Observe that $\rho((i,j))=\rho(-it_{(i,j)})+\rho(+t_{(i,j)}j)$, see Figure \ref{fig:Exchange}.  We use this equality to construct a signed graph $\widetilde{G}$ from $G$, on a possibly larger vertex set, such that the algebraic properties of $k[\widetilde{G}]$ and $k[G]$ are related.

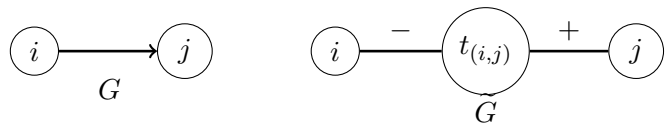
\begin{figure}[bht]
\centering
\begin{tikzpicture}
\node [draw,circle] (A) at (1,0) {$j$};
\node [draw,circle] (B) at (-1,0) {$i$};
\draw [<-,line width=1pt] (A) -- (B);
\node (G) at (0,-.5) {$G$};

\node [draw,circle] (C) at (7,0) {$j$};
\node [draw,circle] (D) at (3,0) {$i$};
\node [draw,circle] (E) at (5,0) {$t_{(i,j)}$};
\draw [line width=1pt] (D) --node[above] {$-$} (E);
\draw [line width=1pt] (C) --node[above] {$+$} (E);
\node (G') at (5,-.75) {$\widetilde{G}$};
\end{tikzpicture}
\caption[Construction of an augmented signed graph from a mixed signed, directed graph.]{The construction of the augmented signed graph $\widetilde{G}$ from the mixed signed, directed graph $G$ replaces directed edges $(i,j)$ with pairs of signed edges $-it_{(i,j)}$ and $+t_{(i,j)}j$.  Since $\rho((i,j))=\rho\left(-it_{(i,j)}\right)+\rho\left(+t_{(i,j)}j\right)$, many algebraic properties, such as normality, of $k[G]$ are preserved in $k[\widetilde{G}]$. \label{fig:Exchange}}
\end{figure}

\begin{definition}\label{def:AugmentedGraph}
Let $G$ be a mixed signed, directed graph.
The {\em augmented signed graph} $\widetilde{G}$ of $G$ is a signed graph where each directed edge $(i,j)$ in $G$ is replaced by a vertex $t_{(i,j)}$ and a pair of edges $-it_{(i,j)}$ and $+t_{(i,j)}j$.  The new vertex $t_{(i,j)}$, adjacent to only $i$ and $j$, is called an {\em artificial vertex}.
\end{definition}

By the observation above, any monomial $x^{\alpha}\in k[G]$ also appears in $k[\widetilde{G}]$ by replacing each use of a directed edge $(i,j)$ with the pair $-it_{(i,j)}$ and $+t_{(i,j)}j$ from $\widetilde{G}$.  Therefore, $k[G]\subseteq k[\widetilde{G}]$, and, similarly $\APG \subseteq \A(\PP_{\widetilde{G}})$.  The following lemma and Figure \ref{fig:CommutativeDiagram} provide details on these inclusion relationships.

\begin{figure}[hbt]
\centering
\begin{tikzpicture}
\matrix (m) [matrix of math nodes,row sep=3em,column sep=4em,minimum width=2em]
  {
     k[G] & \APG \\
     k[\widetilde{G}]\cap k[x_1^{\pm 1},\dots,x_n^{\pm 1}]  & \A(\PP_{\widetilde{G}})\cap k[x_1^{\pm 1},\dots,x_n^{\pm 1}] \\};
  \path[-stealth]
    (m-1-1) edge [right hook->] node[right] {\rotatebox{90}{$=$}} (m-2-1)
            edge [right hook->] (m-1-2)
    (m-2-1) edge [right hook->] (m-2-2)
    (m-1-2) edge [right hook->] node[right]{\rotatebox{90}{$=$}} (m-2-2);
\end{tikzpicture}
\caption[Commutative Diagram for edge rings from an augmented signed graph, and a mixed signed, directed graph.]{Let $G$ be a mixed signed, directed graph and $\widetilde{G}$ the augmented signed graph for $G$.  The commutative diagram illustrates the inclusion relationship and equalities between $k[G]$, $k[\widetilde{G}]$, $\APG$ and $\A(\PP_{\widetilde{G}})$.  The illustrated equalities are proved in Lemma \ref{lem:Equality}.\label{fig:CommutativeDiagram}}
\end{figure}

\begin{lemma}\label{lem:Equality}
Let $G$ be a mixed signed, directed graph with augmented signed graph $\widetilde{G}$.  Consider $k[\widetilde{G}]$ and $\A(\PP_{\widetilde{G}})$ as subrings of $k[x_1^{\pm 1},\dots,x_n^{\pm 1},t_1^{\pm 1},\dots,t_m^{\pm 1}]$.
Then $k[\widetilde{G}]\cap k[x_1^{\pm 1},\dots,x_n^{\pm 1}] = k[G]$ and $\A(\PP_{\widetilde{G}})\cap k[x_1^{\pm 1},\dots,x_n^{\pm 1}] = \APG$.
\end{lemma}
\begin{proof}
The inclusion in one direction has been discussed above.  Let $x^\alpha$ be a monomial in $k[\widetilde{G}]\cap k[x_1^{\pm 1},\dots,x_n^{\pm 1}]$.  Therefore, $\alpha$ can be represented by nonnegative integral edge weights.  Let $(i,j)$ be a directed edge of $G$, then $-it_{(i,j)}$ and $+t_{(i,j)}j$ have the same weight $w_{(i,j)}$ because $e_{t_{(i,j)}}$ has a zero coefficient.  Since $\rho((i,j))=\rho(-it_{(i,j)})+\rho(+t_{(i,j)}j)$, we can assign weights on the edges of $G$ which generate $x^\alpha$: in particular, for any signed edge $e$ of $G$, assign $e$ the same weight as the corresponding signed edge in $\widetilde{G}$, and, for any directed edge $(i,j)$ in $G$, assign $(i,j)$ weight $w_{(i,j)}$.  Since the edge weights are integral, $\alpha\in T_2$ for $G$.

The proof for $\APG$ is identical by after replacing nonnegative integral edge weights with nonnegative real edge weights or integral edge weights.
\end{proof}

\subsection{Normality}
\label{sec:mixedNormality}

We prove another main result in this section, which generalizes the results of Section \ref{sec:Algebraic} to all quadratic-monomial generated domains.  In particular, for all quadratic-generated domains, we prove a complete combinatorial characterization of the normality and provide the generators of the normalization of the edge ring $k[G]$.

\begin{theorem}\label{thm:Digraphs}
Let $G$ be a mixed signed, directed graph and $\widetilde{G}$ the augmented signed graph of $G$, then
$k[G]$ is normal if and only if $\widetilde{G}$ satisfies the odd cycle condition.
\end{theorem}
\begin{proof}
By Theorem \ref{thm:MainResult}, $\widetilde{G}$ satisfies the odd cycle condition if and only if $k[\widetilde{G}]$ is normal, which occurs if and only if $k[\widetilde{G}] = \A(\PP_{\widetilde{G}})$.  By Lemma \ref{lem:Equality}, if $\widetilde{G}$ satisfies the odd cycle condition, then
\[
k[G] = k[\widetilde{G}]\cap k[x_1^{\pm 1},\dots,x_n^{\pm 1}]=\A(\PP_{\widetilde{G}})\cap k[x_1^{\pm 1},\dots,x_n^{\pm 1}]=\APG.
\]
Since $\APG$ is normal, $k[G]$ is normal.

On the other hand, suppose $\widetilde{G}$ does not satisfy the odd cycle condition.
From Theorem \ref{thm:Normalization}, there is a monomial $M_{\Pi}$ in $\A(\PP_{\widetilde{G}})$, but not in $k[\widetilde{G}]$, where $\Pi$ is a pair of exceptional cycles in $\widetilde{G}$.  Moreover, we observe that if $t_{(i,j)}$ is any artifical node in $\widetilde{G}$, then the exponent of $t_{(i,j)}$ in $M_{\Pi}$ is zero since in any cycle $C$ containing $t_{(i,j)}$, $\sig_C\left(t_{(i,j)}\right)=0$.  Therefore, $M_{\Pi}\in \A(\PP_{\widetilde{G}})\cap k[x_1^{\pm 1},\dots,x_n^{\pm 1}]=\APG$, and, similarly, $M_{\Pi}\not\in k[G]$.  Moreover, in Theorem \ref{thm:converse}, we know that $(M_{\Pi})^2\in k[\widetilde{G}]$, so $M_{\Pi}$ is a root of a monic polynomial with coefficients in $k[G]$.  By the proof of Lemma \ref{lem:Equality}, the weights in Theorem \ref{thm:Normalization} can be mapped from $\widetilde{G}$ to $G$, and, so, $k[G]$ is not normal.
\end{proof}

An interesting consequence of Theorem \ref{thm:Digraphs} occurs for directed graphs, i.e., a mixed signed, directed graph which does not have any signed edges.

\begin{corollary}\label{cor:Digraphs}
Let $G$ be a directed graph, then $k[G]$ is normal.
\end{corollary}
\begin{proof}
Observe that the augmented signed graph $\widetilde{G}$ of $G$ is bipartite since every edge of $G$ is replaced with a path of length two.
Hence, $\widetilde{G}$ satisfies the odd cycle condition because $\widetilde{G}$ has no odd cycles.  Thus, by Theorem \ref{thm:Digraphs}, $k[G]=\APG$, which is normal.
\end{proof}

Let $G$ be a mixed signed, directed graph and $\widetilde{G}$ its associated signed graph.  Theorem \ref{thm:Digraphs} implies that the normality of $k[G]$ depends only on odd cycles in $k[\widetilde{G}]$.  Instead of constructing $\widetilde{G}$, we describe, directly, which subgraphs of $G$ produce exceptional odd cycles in $\widetilde{G}$.

Let $C$ be a cycle in $G$ with $k$ signed edges and $\ell$ directed edges.  The corresponding cycle in $\widetilde{G}$ has $k+2\ell$ signed edges.  There is a natural bijection between cycles of $G$ and cycles of $\widetilde{G}$.  Moreover, a cycle in $G$ corresponds to an odd cycle in $\widetilde{G}$ if and only if it has an odd number of signed edges.  

We can also describe alternating paths in $G$.  Since a directed edge in $G$ is replaced with an alternating path of length two in $\widetilde{G}$, an alternating path is one where the signed edges are alternating, adjacent directed edges must point in the same direction, positive edges are incident to the tails of directed edges, and negative edges are incident to the heads of directed edges.

\begin{definition}\label{def:genaltpath}
Suppose $P$ is a sequence of incident edges in $G$ with vertex set $\{i_0,i_1,\dots,i_n\}$.
We say $P$ is a {\em generalized alternating path} if:
\begin{itemize}
\item The subsequence of signed edges obtained by deleting the directed edges alternates,
\item If $(i_{a-1},i_a)$ is a directed edge, then $\sgn(i_{a-2}i_{a-1})=+1$ and $\sgn(i_ai_{a+1})=-1$ for the signed edges $i_{a-2}i_{a-1}$ or $i_ai_{a+1}$ preceding or following the directed edge, if they exist.
\item If $(i_a,i_{a-1})$ is a directed edge, then $\sgn(i_{a-2}i_{a-1})=-1$ and $\sgn(i_ai_{a+1})=+1$ for the signed edges $i_{a-2}i_{a-1}$ or $i_ai_{a+1}$ preceding or following the directed edge, if they exist.
\item If $(i_{a-1},i_a)$ and $(i_b,i_{b-1})$ are directed edges pointing in opposite directions in $P$ then there are an odd number of signed edges between them,
\item If $(i_{a-1},i_a)$ and $(i_{b-1},i_b)$ are directed edges pointing in the same direction in $P$ then there are an even number of signed edges between them.
\end{itemize}
\end{definition}

\begin{definition}
Let $G$ be a mixed signed, directed graph, we say that $G$ satisfies the {\em generalized odd cycle condition} if for every pair of cycles $C$ and $C'$ of $G$ where $C$ and $C'$ both have an odd number of signed edges, at least one of the following occurs: 
\begin{enumerate}
\item $C$ and $C'$ are in distinct components,
\item $C$ and $C'$ have a vertex in common,
\item $C$ and $C'$ have a generalized alternating $i,j$-path in $G$, where $i$ is a vertex of $C$ and $j$ is a vertex of $C'$.
\end{enumerate}
\end{definition}

We compute the normalization of $k[G]$ from the normalization of $k[\widetilde{G}]$:

\begin{theorem}\label{thm:mixedNormalization}
Let $G$ be a mixed signed, directed graph, and $\widetilde{G}$ the associated augmented signed graph.
If the normalization of $k[\widetilde{G}]$, $\mathcal{A}(\mathcal{P}_{\widetilde{G}})$, is generated by monomials $M_{\Pi_1},\dots,M_{\Pi_m}$, over $k[\widetilde{G}]$, then the normalization of $k[G]$, $\APG$, is generated by $M_{\Pi_1},\dots,M_{\Pi_m}$ over $k[G]$.
\end{theorem}
                        
\begin{proof}
By the argument in Theorem \ref{thm:Digraphs}, we know that each $M_{\Pi_\ell}$ is in $\A(\PP_{\widetilde{G}})\cap k[x_1^{\pm 1},\dots,x_n^{\pm 1}]=\APG$.  Moreover, we know, from Theorem \ref{thm:Normalization}, that the $M_{\Pi_\ell}$'s generate $\A(\PP_{\widetilde{G}})$ over $k[\widetilde{G}]$.  Our goal is to show that these monomials generate $\APG$ over $k[G]$.

We observe that $k[\widetilde{G}]$ and $\A(\PP_{\widetilde{G}})$ are monomial-generated since each $\rho(e)$ and $M_{\Pi}$ is a monomial.  Therefore, $\alpha\in\A(\PP_{\widetilde{G}})$ if and only if each of its terms are in $\A(\PP_{\widetilde{G}})$.  A similar statement holds for $k[\widetilde{G}]$.  Let $\alpha\in\APG$, then $\alpha\in\A(\PP_{\widetilde{G}})$.  Therefore, each term of $\alpha$ is in $\A(\PP_{\widetilde{G}})$, and let $\alpha_1$ a term of $\alpha$.  Since $\A(\PP_{\widetilde{G}})$ is generated as an algebra over $k[\widetilde{G}]$ by the $M_{\Pi_\ell}$'s, we know that $\alpha_1=\beta\prod M_{\Pi_{\ell_j}}$ where $\beta\in k[\widetilde{G}]$.  Since none of the artificial vertices appear in $M_{\Pi_\ell}$ for any $\ell$ or in $\alpha_1$, it follows that $\beta\in k[x_1^{\pm 1},\dots,x_n^{\pm 1}]$, and, so, by Lemma \ref{lem:Equality}, the equation for $\alpha_1$ is true in $k[G]$.  Therefore, $\APG$ is generated over $k[G]$ by the $M_{\Pi_\ell}$'s.
\end{proof}

\section{Conclusion}\label{sec:conclusion}

In this paper, we have extended the normality characterization of edge rings, originally presented by Hibi and Ohsugi \cite{OhsugiHibi:1998} and Simis, Vasconcelos, and Villarreal \cite{Villarreal:1998}, to arbitrary mixed signed, directed graphs.  This results in a  complete characterization of the normality and the normalization of quadratical-monomial generated polynomial rings in terms of the combinatorial properties of the associated graphs.  The proofs in this paper are new and have additional subtleties when compared to \cite{OhsugiHibi:1998} and \cite{Villarreal:1998} due to the possibility of cancellation.

\section*{Acknowledgements}\label{sec:Acknowledgements}
This work was partially supported by a grant from the Simons Foundation (\#282399 to Michael Burr) and NSF Grant CCF-1527193.
\bibliography{bibliographyNormalDomains}
\bibliographystyle{spmpsci}
%\begin{thebibliography}{}
% etc
%\end{thebibliography}

\end{document}